\documentclass[11pt]{article}
\usepackage[dvips]{graphicx,color}
\usepackage{amssymb}
\usepackage{amsthm}

\newtheorem{theorem}{Theorem}[section]
\newtheorem{lemma}{Lemma}[section]
\newtheorem{corollary}{Corollary}[section]
\newtheorem{remark}{Remark}[section]

\begin{document}
\title{A method for computing the Perron root for primitive
matrices}

\author{Doulaye Demb\'el\'e\footnote{Email: doulaye@igbmc.fr}}
\date{\small Institut de G\'en\'etique et de Biologie
Mol\'eculaire et Cellulaire (IGBMC),\\
CNRS UMR 7104, INSERM U1258, Universit\'e de Strasbourg\\
1 rue Laurent Fries, 67400, Illkirch-Graffenstaden, France}
\maketitle

\begin{abstract}
Following the Perron theorem, the spectral radius of a primitive
matrix is a simple eigenvalue. It is shown that for a primitive matrix
$A$, there is a positive rank one matrix $X$ such that $B = A \circ X$,
where $\circ$ denotes the Hadamard product of matrices, and such that
the row (column) sums of matrix $B$ are the same and equal to the
Perron root. An iterative algorithm is presented to obtain matrix
$B$ without an explicit knowledge of $X$. The convergence rate of this
algorithm is similar to that of the power method but it uses less
computational load. A byproduct of the proposed algorithm is a new method 
for calculating the first eigenvector.

\textit{Keywords:}
primitive matrix; Perron root; Markov chain; stochastic matrix.

\textit{MSC(2010):}
 15A18\ 15A48\ 15A03\ 65F10\ 65F15\ 65C40
\end{abstract}

\section{Introduction}
\label{sect_introduction}
Given a nonnegative matrix, the problem of computing the first eigenvalue 
and eigenvector is considered in this paper.
For two matrices $A=(a_{ij})$ and $B=(b_{ij})$ with the same number of rows
and columns, their Hadamard product is a matrix of elementwise products:
\begin{equation}
A\circ B = (a_{ij}b_{ij}) \label{eq_hadamard}
\end{equation}
For scalars $\alpha$ and $\beta$:
\begin{equation}
\alpha A \circ \beta B = \alpha\beta(A\circ B) \label{eq_hadamard2}
\end{equation}
If $A$ and $B$ are rank one matrices, i.e. 
$A=\mathbf{u}\mathbf{v}^T$ and $B=\mathbf{x}\mathbf{y}^T$ then
\begin{equation}
A\circ B = (\mathbf{u}\mathbf{v}^T)\circ (\mathbf{x}\mathbf{y}^T) 
       = (\mathbf{u}\circ \mathbf{x})(\mathbf{v}\circ \mathbf{y})^T. 
       \label{eq_hadamard3}
\end{equation}
Many properties for Hadamard product are given in 
\cite{Styan_1973},\cite[chapter 5]{Horn_al_1991}.

If $A=(a_{ij})\in\mathbb{R}^{n\times n}$, then $A$ is called \textit{positive}
if $a_{ij}>0$, and \textit{nonnegative} if $a_{ij}\geq 0$. 
Perron \cite{Perron_1907} showed that the spectral radius of a positive
matrix $A$ is a simple eigenvalue \cite[page 667]{Meyer_2000} that
dominates all other eigenvalues in modulus. This eigenvalue, denoted
$\rho(A)$, is called the \textit{Perron root} and the associated normalized
positive vector is called the \textit{Perron vector}.
Nonnegative matrices are frequently encountered in
real life applications \cite{Brualdi_al_1991,Berman_al_1994}.
Frobenius \cite{Frobenius_1912} extented Perron's
work on positive matrices to nonnegative matrices. 
The spectral radius of a nonnegative matrix $A$ is positive if it is
irreducible, i.e. $(I_n+A)^{n-1}$ is a positive matrix
\cite[page 534]{Horn_al_2019}, \cite[page 672]{Meyer_2000}.
The dominant eigenvalue of an irreducible matrix is unique if it is
primitive \cite[page 540]{Horn_al_2019}, \cite[page 674]{Meyer_2000}.
A nonnegative matrix is primitive if $A^m$ is a positive matrix for
some non nul $m$ \cite[page 540]{Horn_al_2019}, \cite[page 678]{Meyer_2000}.
Wielandt, \cite{Wielandt_1950}, showed that a nonnegative matrix
$A$ of order $n$ is primitive if $A^{n^2-2n+2}$ is a positive matrix
\cite[page 543]{Horn_al_2019}. To verify primitivity
of a nonnegative matrix using Frobenius or Wielandt formula leads to huge
calculations especially when $n$ is high. It is shown in 
\cite[page 544]{Horn_al_2019} that only some power calculations of the matrix
are necessary.

This paper is on the calculation of the Perron root.
The power method is generally used to obtain the eigenvalue
with the maximum modulus and associated eigenvector 
\cite[page 545]{Horn_al_2019},
\cite[page 330]{Golub_al_1996}, \cite[page 533]{Meyer_2000}.
The convergence rate of the power method
depends on the ratio of the second eigenvalue to the first
\cite[page 330]{Golub_al_1996}, \cite[page 533]{Meyer_2000}.
More iterations will be required when the modulus of the second highest
eigenvalue is close to that of the first. Here, an iterative algorithm is
proposed for calculating the Perron root for primitive matrices.
This algorithm is based on successive improvement of bounds for the Perron
root. There are many research works on localization of the Perron root for
nonnegative matrices \cite{Kolotilina_1993,Liu_1996,
Duan_al_2013,Xing_al_2014,Liao_2017}. 
Frobenius carried out the following bounds \cite[page 521]{Horn_al_2019}:
\begin{eqnarray}
 \min_{i=1,\ldots,n}\left\{r_i(A)\right\} \leq \rho(A) \leq
      \max_{i=1,\ldots,n}\left\{r_i(A)\right\} \label{eq_bounds_frob1}\\
 \min_{j=1,\ldots,n}\left\{c_j(A)\right\} \leq \rho(A) \leq
      \max_{j=1,\ldots,n}\left\{c_j(A)\right\} \label{eq_bounds_frob2}
\end{eqnarray}
where $r_i(A) = \sum_{j=1}^na_{ij}$ and $c_j(A) = \sum_{i=1}^na_{ij}$
are the row and column sums of $A$, respectively. In (\ref{eq_bounds_frob1}) 
and (\ref{eq_bounds_frob2}), equalities occur when $\rho(A)$ is equal to the
row or column sums. The column sums of the matrix in (\ref{eq_A_example})
are both equal to $3$.
\begin{equation}
A = \left(\begin{array}{ccc}
              0 &1 &0\\
              3 &0 &3\\
              0 &2 &0
\end{array}\right)  \label{eq_A_example}
\end{equation}
The matrix $A$ in (\ref{eq_A_example}) is imprimitive and its
eigenvalues are: $3$, $-3$ and $0$.
This example shows that equality in (\ref{eq_bounds_frob1}) or 
(\ref{eq_bounds_frob2}) can occur for an imprimitive matrix.

Let $\textbf{x}=(x_1,x_2,\ldots,x_n)$ a vector with only positive values,
$x_i>0$, and $D_{\mathbf{x}}$ a diagonal matrix formed with $\textbf{x}$.
The matrix $B$ defined by:
\begin{equation}
B=D_{\mathbf{x}}^{-1}AD_{\mathbf{x}} \label{eq_similar}
\end{equation}
is diagonally similar to $A$ \cite{Elsner_al_1988}, and we have:

\begin{lemma} \label{lemma1}The matrices $A$ and $B$ in (\ref{eq_similar})
have the same eigenvalues, and
\begin{itemize}
\item[a)] if $A$ is irreducible, then $B$ is irreducible,
\item[b)] if $A$ is primitive, then $B$ is primitive,
\end{itemize}
\end{lemma}

\begin{proof}
 a) if $A$ is irreducible then $(I_n+A)^{n-1}$ is a positive matrix.
    From (\ref{eq_similar}), we have:
    \begin{eqnarray}
    I_n+B &=& I_n+D_{\mathbf{x}}^{-1}AD_{\mathbf{x}} 
          = D_{\mathbf{x}}^{-1}(I_n+A)D_{\mathbf{x}}\\
    (I_n+B)^{n-1} &=& D_{\mathbf{x}}^{-1}(I_n+A)^{n-1}D_{\mathbf{x}}
    \end{eqnarray}
$D_{\mathbf{x}}$ has only positive values and $(I_n+A)^{n-1}$ is a positive
matrix. The matrix $(I_n+B)^{n-1}$ is then positive that implies 
irreducibility of the matrix $B$.

b) if $A$ is a primitive matrix then there exists an integer $m$ such that
$A^m$ is positive. Using (\ref{eq_similar}) we have 
$B^m=D_{\mathbf{x}}^{-1}A^mD_{\mathbf{x}}$, that implies
$B^m$ is positive and the result follows.
\end{proof}

Using an improvement of bounds in (\ref{eq_bounds_frob1}) and 
(\ref{eq_bounds_frob2}) by Minc \cite{Minc_1988}, relation 
(\ref{eq_similar}) and the uniqueness of eigenvalue with a maximum modulus 
for a primitite matrix, an iterative algorithm is proposed to obtaining the
Perron root.

\section{Methods\label{sect_methods}}
\begin{lemma}\label{lem_zero_row_col}
Let $A=(a_{ij})\in\mathcal{R}^{n\times n}$ a nonnegative matrix.
If matrix $A$ has a row (column) with only zero entries, then $A$ cannot
be a primitive matrix.
\end{lemma}  
\begin{proof}
see an exercise on irreducible matrices in \cite[page 522]{Horn_al_2019}.
\end{proof}

For a primitive matrix, from the Lemma \ref{lem_zero_row_col} and
relations (\ref{eq_bounds_frob1})-(\ref{eq_bounds_frob2}), we have the
following two observations. The minimum value of the row (column) sums for
a primitive matrix is greater than zero. The maximum value of the row 
(column) sums for a primitive matrix is greater than or equal to the 
Perron root.

Let us note $D_{\mathbf{r}}$ and $D_{\mathbf{c}}$ diagonal matrices formed
with the row sums $\mathbf{r} = (r_1(A), r_2(A), \ldots, r_n(A))$ and the
column sums $\mathbf{c} = (c_1(A), c_2(A), \ldots, c_n(A))$ of A,
respectively. The Frobenius bounds (\ref{eq_bounds_frob1}) and 
(\ref{eq_bounds_frob2}) have been improved by Minc \cite[page 27]{Minc_1988}:
\begin{eqnarray}
    \min_{i=1,\ldots,n}\left\{r_i(D_{\mathbf{r}}^{-1}AD_{\mathbf{r}})\right\} 
    \leq \rho(A) \leq \max_{i=1,\ldots,n}
    \left\{r_i(D_{\mathbf{r}}^{-1}AD_{\mathbf{r}})\right\}
    \label{eq_bounds_minc1}\\
    \min_{j=1,\ldots,n}\left\{c_j(D_{\mathbf{c}}^{-1}AD_{\mathbf{c}})\right\} 
    \leq \rho(A) \leq \max_{j=1,\ldots,n}
    \left\{c_j(D_{\mathbf{c}}^{-1}AD_{\mathbf{c}})\right\}
    \label{eq_bounds_minc2}
\end{eqnarray}

In (\ref{eq_bounds_minc1}) and (\ref{eq_bounds_minc2}), equalities hold when
the row or column sums are the same and correspond to the
Perron root. Using the row sums relation, (\ref{eq_bounds_minc1})
allows to write:
\begin{eqnarray}
D_{\mathbf{r}}^{-1}AD_{\mathbf{r}} &=& \left(\begin{tabular}{ccccc}
    $a_{11}$ &$\frac{r_2}{r_1}a_{12}$ &$\frac{r_3}{r_1}a_{13}$ &\ldots
             &$\frac{r_n}{r_1}a_{1n}$\\
    $\frac{r_1}{r_2}a_{21}$ &$a_{22}$ &$\frac{r_3}{r_2}a_{23}$ &\ldots
             &$\frac{r_n}{r_2}a_{2n}$\\
    $\frac{r_1}{r_3}a_{31}$ &$\frac{r_2}{r_3}a_{32}$ &$a_{33}$ &\ldots
             &$\frac{r_n}{r_3}a_{3n}$\\
    $\vdots$ &$\vdots$ &$\vdots$ &$\ddots$ &$\vdots$\\
    $\frac{r_1}{r_n}a_{n1}$ &$\frac{r_2}{r_n}a_{n2}$ &$\frac{r_3}{r_n}a_{n3}$
             &\ldots &$a_{nn}$
\end{tabular}\right) \label{eq_Dm1ADa}\\
 &=& \left(\begin{tabular}{ccccc}
    $a_{11}$ &$a_{12}$ &$a_{13}$ &\ldots &$a_{1n}$\\
    $a_{21}$ &$a_{22}$ &$a_{23}$ &\ldots &$a_{2n}$\\
    $a_{31}$ &$a_{32}$ &$a_{33}$ &\ldots &$a_{3n}$\\
    $\vdots$ &$\vdots$ &$\vdots$ &$\ddots$ &$\vdots$\\
    $a_{n1}$ &$a_{n2}$ &$a_{n3}$ &\ldots &$a_{nn}$
    \end{tabular}\right) \circ \left(\begin{tabular}{ccccc}
    1 &$\frac{r_2}{r_1}$ &$\frac{r_3}{r_1}$ &\ldots
             &$\frac{r_n}{r_1}$\\
    $\frac{r_1}{r_2}$ &1 &$\frac{r_3}{r_2}$ &\ldots
             &$\frac{r_n}{r_2}$\\
    $\frac{r_1}{r_3}$ &$\frac{r_2}{r_3}$ &1 &\ldots
             &$\frac{r_n}{r_3}$\\
    $\vdots$ &$\vdots$ &$\vdots$ &$\ddots$ &$\vdots$\nonumber\\
    $\frac{r_1}{r_n}$ &$\frac{r_2}{r_n}$ &$\frac{r_3}{r_n}$
             &\ldots &1
\end{tabular}\right)\\
 &=& A \circ X\label{eq_mat_Dm1AD}
\end{eqnarray}
where $X$ is a positive matrix formed with:
\begin{equation}
x_{ij} = \frac{r_j(A)}{r_i(A)}\mbox{ ; } i,j=1,2\ldots,n \label{eq_x_values}
\end{equation}

The unicity of the Perron root for a primitive matrix and
(\ref{eq_mat_Dm1AD}) suggest that the components of the matrix $X$ can be
chosen to have the same row sums for $A\circ X$. For a second
order nonnegative matrix ($n=2$), we have:
\begin{equation}
A\circ X = \left(\begin{array}{cc}
       a_{11} &a_{12}x\\
       a_{21}/x &a_{22}
       \end{array}\right)\label{eq_2order}
\end{equation}
where $x=r_2(A)/r_1(A)$, $r_1(A)=a_{11}+a_{12}$ and
$r_2(A)=a_{21}+a_{22}.$

If the row sums in (\ref{eq_2order}) are the same and equal to $S$, an
expression can be obtained for the parameter $x$:
\begin{equation}
x = \frac{S-a_{11}}{a_{12}} \mbox{ ; }
\frac{1}{x} = \frac{S-a_{22}}{a_{21}} \label{eq_2order_x}
\end{equation}
To have a value for $x$, $a_{12}$ and $a_{21}$ should be nonzero. Relation
(\ref{eq_2order_x}) allows to have a second order equation which resolution
leads to a value for $S$:
\begin{equation}
S^2 - (a_{11}-a_{22})S + a_{11}a_{22}-a_{12}a_{21} = 0 \label{eq_2order_S}
\end{equation}
The solution of (\ref{eq_2order_S}) with the maximum modulus is:
\begin{equation}
S = \left(a_{11}+a_{22}+\sqrt{(a_{11}-a_{22})^2+4a_{12}a_{21}}
  \right)/2 \label{eq_2order_root}
\end{equation}
A second order nonnegative matrix $A$ is primitive if $a_{12}$, and $a_{21}$
are both nonzero, on the one hand. On the other hand, at least $a_{11}$
or $a_{22}$ should be nonzero. Hence, for a second order primitive matrix, 
explicit expressions can be obtained for matrix $X$ (parameter $x$) and the 
Perron root, (\ref{eq_2order_root}). However, a direct search for
components of the matrix $X$ in (\ref{eq_mat_Dm1AD}) becomes difficult
when $n>2$. 

\begin{lemma}\label{lem_y_rhoA}
Let $A=(a_{ij})$ a square matrix of order $n$,
$\mathbf{y}=\alpha\mathbf{x}$ a vector where $\alpha$ is a nonzero scalar
and $\mathbf{x}$ is the eigenvector associated with eigenvalue $\lambda$ of 
$A$. If all components of $\mathbf{x}$ have nonzero value, the row sums of 
the matrix $D_{\mathbf{y}}^{-1}AD_{\mathbf{y}}$ are the same and equal to
the eigenvalue $\lambda$ of $A$.

A similar result is obtained using $A^T$ or the column sums.
\end{lemma}

\begin{proof}
Using the definition of the eigenvalue, the component $i$ of
$A\mathbf{x} = \lambda\mathbf{x}$ is:
\begin{equation}
\sum_{j=1}^{n}a_{ij}x_j = \lambda x_i \label{eq_lem_y_rhoA1}
\end{equation}
The component $i$ of the row sums of the matrix
$D_{\mathbf{y}}^{-1}AD_{\mathbf{y}}$ is:
\begin{equation}
r_i=\frac{1}{x_i}\left(\sum_{j=1}^{n}a_{ij}x_j\right)\label{eq_lem_y_rhoA2}
\end{equation}
By replacing the right hand expression of (\ref{eq_lem_y_rhoA1})
in (\ref{eq_lem_y_rhoA2}) the result follows.
\end{proof}

Compared to the Frobenius bounds (\ref{eq_bounds_frob1})
and (\ref{eq_bounds_frob2}), the bounds in (\ref{eq_bounds_minc1}) and
\ref{eq_bounds_minc2} are based on a modification of the initial matrix.
This process can be repeated to further sharpen the bounds. From the unicity
of the Perron root for a primitive matrix, a repeative improvement
of the Minc bounds will lead to equalities of the row (column) sums.
The main result of this paper is the following.

\begin{theorem}\label{theo_AoX}
Let $A$ be a primitive matrix of order $n$. There exists a positive rank 
one matrix $X$ of the form
\begin{equation}
X = \left(\begin{array}{ccccc}
1 &x_2/x_1 &x_3/x_1 &\ldots &x_n/x_1\\
x_1/x_2 &1 &x_3/x_2 &\ldots &x_n/x_2\\
x_1/x_3 &x_2/x_3 &1 &\ldots &x_n/x_3\\
\vdots &\vdots &\vdots &\ddots &\vdots\\
x_1/x_n &x_2/x_n &x_3/x_n &\ldots &1
\end{array}\right)  \label{eq_mat_X}
\end{equation}
such that the matrix $B=A\circ X$ is similar to $A$. In addition, the
row (column) sums of $B$ are the same and equal to the Perron
root of $A$.
\end{theorem}

\begin{proof}
The row sums are used for the proof, the column sums can also be used in a
similar way.

Let us note
$A^{(0)}$ the initial matrix and its row sums vector as $\mathbf{r}^{(0)}$.
Relation (\ref{eq_bounds_minc1}) allows to write:
\begin{equation}
A^{(t)} = D_{\mathbf{r}^{(t-1)}}^{-1}A^{(t-1)}D_{\mathbf{r}^{(t-1)}}
        \mbox{ ; } t=1, 2, \ldots \label{eq_mat_At}
\end{equation}
From (\ref{eq_mat_Dm1AD}) and (\ref{eq_x_values}), the components of
matrix $A^{(t)}$ and its row sums are:
\begin{eqnarray}
a_{ij}^{(t)} &=& \frac{r_j^{(t-1)}(A^{(t-1)})}{r_i^{(t-1)}(A^{(t-1)})}
             a_{ij}^{(t-1)}\mbox{ ; } i,j=1,2,\ldots,n\label{eq_at_values}\\
r_i^{(t)}(A^{(t)}) &=&  \sum_{j=1}^n  a_{ij}^{(t)}
         \mbox{ ; } i=1, 2, \ldots, n\label{eq_rit_values}
\end{eqnarray}
At iteration $t$, $A^{(t)}$ is similar to $A^{(t-1)}$. From the Minc relation
(\ref{eq_bounds_minc1}), the bounds with $A^{(t)}$ are improved compared to
those with $A^{(t-1)}$. Hence, when $t\to \infty$, equalies hold in
(\ref{eq_bounds_minc1}) for primitive matrix and
the row sums $r_i^{(t)}(A^{(t)}), i=1, 2, \ldots,n$,
have the same value which is equal to $\rho(A)$, 
thank to Lemma \ref{lem_y_rhoA}.

From (\ref{eq_mat_At}) and the notation in (\ref{eq_mat_Dm1AD}), 
we can write:
\begin{eqnarray}
 A^{(1)} &=& A^{(0)}\circ X^{(0)} = A\circ X^{(0)}\\
 A^{(2)} &=& A^{(1)}\circ X^{(1)} = A\circ X^{(0)}\circ X^{(1)}\\
   &\vdots& \nonumber\\
 A^{(t)} &=& A\circ X^{(0)}\circ X^{(1)}\circ\ldots\circ 
         X^{(t-2)}\circ X^{(t-1)} \label{eq_A1_t}
\end{eqnarray}
At the convergence iteration $t$, $X^{(t)}$ is formed
with only $1$. Then, from (\ref{eq_A1_t}) we have
\begin{equation}
B = A^{(t)} = A\circ X  \label{eq_AoX}
\end{equation}
where:
\begin{eqnarray}
X      &=& X^{(0)}\circ X^{(1)}\circ\ldots\circ X^{(t-2)}\circ X^{(t-1)}
           \label{eq_mat_X2}\\
x_{ij} &=& \prod_{s=0}^{t-1}\frac{r_j^{(s)}(A^{(s)})}{r_i^{(s)}(A^{(s)})}
       \mbox{ ; } i,j=1,2,\ldots,n
\end{eqnarray}
Relation (\ref{eq_AoX}) is another form of (\ref{eq_Dm1ADa}), then matrix
$B$ is similar to matrix $A$. Since the row sums of matrix $B$ are the same,
equalities occur in (\ref{eq_bounds_frob1}) and lead to the Perron root.

The matrix $X$ in (\ref{eq_mat_X}) can be write as a product of two vectors:
\begin{equation}
X = \mathbf{x}\mathbf{y}^T  \label{eq_mat_X3}
\end{equation}
where $\mathbf{x} = (1, x_1/x_2, \ldots, x_1/x_n)^T$
and $\mathbf{y} = (1, x_2/x_1, \ldots, x_n/x_1)^T$, i.e.
the first column and first row of the matrix $X$, respectively.
$X$ is then a rank one matrix, and is also positive
because formed with the row sums of a primitive matrix,
see Lemma \ref{lem_zero_row_col}.
\end{proof}

\begin{corollary}\label{corollary_1}
Vector $\mathbf{y}$ allowing to obtain the matrix $X$ in
(\ref{eq_mat_X3}) is in the space spanned by the Perron vector of
a primitive matrix $A$.
\end{corollary}
\begin{proof}
Use (\ref{eq_Dm1ADa}), (\ref{eq_mat_Dm1AD}) and Lemma \ref{lem_y_rhoA}.
\end{proof}

\subsection{Convergence of the proposed algorithm}
\begin{theorem}\label{theo_convergence}
The iterative algorithm based on a successive improvement of the Minc bounds 
is convergent for a primitive matrix.
\end{theorem}

\begin{proof}
At iteration $t$, the row sum vectors associated with matrices $A^{(t)}$ and
$A^{(t-1)}$ are $\mathbf{r}^{(t)}(A^{(t)})$ and
$\mathbf{r}^{(t-1)}(A^{(t-1)})$, respectively.
From the Minc theorem \cite[page 27]{Minc_1988}, we have:
\begin{eqnarray}
\min_i\left\{r_i^{(t)}(A^{(t)})\right\} &\geq&
       \min_i\left\{r_i^{(t-1)}(A^{(t-1)})\right\} \label{eq_minc1}\\
\max_i\left\{r_i^{(t)}(A^{(t)})\right\} &\leq&
       \max_i\left\{r_i^{(t-1)}(A^{(t-1)})\right\} \label{eq_minc2}
\end{eqnarray}
Let us define two decreasing sequences as follows:
\begin{eqnarray}
\xi^{(t)} &=&\rho(A) - \min_i\left\{r_i^{(t)}(A^{(t)})\right\}
          \label{eq_minc3}\\
\zeta^{(t)} &=& \max_i\left\{r_i^{(t)}(A^{(t)})\right\} - \rho(A)
            \label{eq_minc4}
\end{eqnarray}
From (\ref{eq_minc2}) and (\ref{eq_minc4}), we have
\begin{eqnarray}
\zeta^{(t)} &=& \max_i\left\{r_i^{(t)}(A)\right\} - \rho(A) 
            \leq \max_i\left\{r_i^{(t-1)}(A)\right\} -
            \rho(A)= \zeta^{(t-1)} \label{eq_minc5}\\
\zeta^{(t)} &\leq& c^{(t)} \zeta^{(t-1)} \label{eq_minc6}
\end{eqnarray}
where $0<c^{(t)} \leq 1$. Let $\zeta^{(0)}$ denotes the initial value
obtained using (\ref{eq_minc4}). From relation (\ref{eq_minc6}) we have:
\begin{equation}
\zeta^{(t)} \leq \left(\prod_{i=1}^tc^{(i)}\right) \zeta^{(0)} 
            \label{eq_minc7}
\end{equation}
Since $c^{(i)}$, $i=1,2,\ldots,t$, are positive numbers not all equal 
to $1$, we have
\begin{equation}
\lim_{t\to\infty}\zeta^{(t)} = 0  \label{eq_minc8}
\end{equation}
A similar reasoning using (\ref{eq_minc1}) and (\ref{eq_minc3}) leads to
$\lim_{t\to\infty}\xi^{(t)} = 0$.
Hence, when the number of iteration goes to infinity, relations 
(\ref{eq_minc3}) and (\ref{eq_minc4}) show that the row sums obtained with 
the algorithm converges to the Perron root.
Referring to Lemma \ref{lem_y_rhoA}, it is like a vector with only ones is 
used to obtaining to Perron root.
\end{proof}

Relation (\ref{eq_minc7}) can be used to estimate the minimum number of 
iterations required by the algorithm before convergence when an error 
level $\alpha$ is set. Assuming that $E(c^{(i)})=c$ is the mean of the
$c^{(i)}$ coefficients, relation (\ref{eq_minc7}) becomes
$\zeta^{(t)} = c^t\zeta^{(0)}$ and we have
\begin{equation}
c^t \leq \alpha \Rightarrow t\geq \frac{\log(\alpha)}{\log(c)}
\end{equation} 

\begin{corollary}\label{corollary_2}
For primitive matrices,
the convergence rate of the proposed algorithm is similar to that of the
power method and depends on the magnitude of the second highest eigenvalue.
\end{corollary}

\begin{proof}
The matrix $A$ can be write as the sum of rank one matrices using its
eigenvalues and eigenvectors:
\begin{eqnarray}
A &=&U\Lambda U^{-1} = U\Lambda V^T\label{eq_mat_A_eig}\\
  &=& \lambda_1\mathbf{u}_1\mathbf{v}_1^T
    + \lambda_2\mathbf{u}_2\mathbf{v}_2^T
    + \ldots \lambda_n\mathbf{u}_n\mathbf{v}_n^T\label{eq_mat_A}
\end{eqnarray}
From (\ref{eq_mat_A}), (\ref{eq_mat_X3}) and (\ref{eq_hadamard3}) we have:
\begin{equation}
A\circ X = \lambda_1(\mathbf{u}_1\circ \mathbf{x})
         (\mathbf{v}_1\circ \mathbf{y})^T
    + \lambda_2(\mathbf{u}_2\circ \mathbf{x})(\mathbf{v}_2\circ \mathbf{y})^T
    + \ldots \lambda_n\mathbf{u}_n\circ \mathbf{x})
      (\mathbf{v}_n\circ \mathbf{y})^T\label{eq_mat_AX}
\end{equation}
Using (\ref{eq_mat_A_eig}) and (\ref{eq_mat_A}), an expression for power $k$
of matrix $A$ is
\begin{equation}
A^k = \lambda_1^k\mathbf{u}_1\mathbf{v}_1^T
    + \lambda_2^k\mathbf{u}_2\mathbf{v}_2^T
    + \ldots \lambda_n^k\mathbf{u}_n\mathbf{v}_n^T\label{eq_mat_Ak}
\end{equation}
This expression allows to have another one similar to
(\ref{eq_mat_AX}).
Then, the row sums at the first step of the algorithm using $A^k$ are:
\begin{eqnarray}
(A^{k}\circ X^{(0)})\mathbf{1} &=& \lambda_1^k\delta_1\mathbf{z}_1 +
            \lambda_2^k\delta_2\mathbf{z}_2 + \ldots + 
            \lambda_n^k\delta_n\mathbf{z}_n\\
            &=& \lambda_1^k\delta_1\left(\mathbf{z}_1 +
            \frac{\delta_2}{\delta_1}\left(\frac{\lambda_2}
            {\lambda_1}\right)^k\mathbf{z}_2 + \ldots +
            \frac{\delta_n}{\delta_1}\left(\frac{\lambda_n}
            {\lambda_1}\right)^k\mathbf{z}_n
            \right)\label{eq_mat_Ak2}
\end{eqnarray}
where $\delta_i=(\mathbf{v}_i\circ \mathbf{y}^{(0)})^T\mathbf{1}$,
$\mathbf{z}_i=(\mathbf{u}_i\circ \mathbf{x}^{(0)})$,
$i=1,2,\ldots,n$,

$\mathbf{x}^{(0)} = \left(1, \frac{r_1^{(0)}}{r_2^{(0)}},
                \frac{r_1^{(0)}}{r_3^{(0)}}, \ldots,
                \frac{r_1^{(0)}}{r_n^{(0)}}\right)^T$ and
$\mathbf{y}^{(0)} = \left(1, \frac{r_2^{(0)}}{r_1^{(0)}},
                \frac{r_3^{(0)}}{r_1^{(0)}}, \ldots,
                \frac{r_n^{(0)}}{r_1^{(0)}}\right)^T$

From corollary \ref{corollary_1}, vector
$(A\circ X)\mathbf{1}\in span\left\{(A^k\circ X^{(0)})\mathbf{1}\right\}$
then,
\begin{equation}
dist\left(span(A\circ X)\mathbf{1}, span(\mathbf{z}_1)\right)
=\mathcal{O}\left(\left|\frac{\lambda_2}{\lambda_1}\right|^k\right)
\end{equation}
\end{proof}

\subsection{Algorithm and implementation}
Relations (\ref{eq_at_values}) and (\ref{eq_rit_values}) allow to obtain
an algorithm for computing the matrix $B$ in Theorem \ref{theo_AoX}. A
convergence test is based on the difference between the maximum and the 
minimum values of the row or column sums, i.e. the range value.
\begin{eqnarray}
error &=& \max_{i=1,\ldots,n}\left\{r_i^{(t)}(A^{(t)})\right\} -
        \min_{i=1,\ldots,n}\left\{r_i^{(t)}(A^{(t)})\right\}\label{eq_error}\\
error &=& \max_{j=1,\ldots,n}\left\{c_j^{(t)}(A^{(t)})\right\} -
        \min_{j=1,\ldots,n}\left\{c_j^{(t)}(A^{(t)})\right\}\label{eq_error2}
\end{eqnarray}
Another convergence test can be based on the examination of the minimum
and maximum row sum values, i.e. by using the decreasing sequences
$\xi^{(t)}$ and $\zeta^{(t)}$ defined in Theorem \ref{theo_convergence}.

\subsubsection{Algorithm A (using row sums): Perron root only
\label{algo_1}}
\begin{enumerate}
\item Initialization
      \begin{itemize}
      \item set: $t\leftarrow 0$, $a_{ij}^{(t)} \leftarrow a_{ij}$,
            calculate the row sums using (\ref{eq_rit_values})
      \item set stopping rules: $eps$ (the acceptable error),
            $maxIter$ (the maximum number of iterations), compute
            the initial error value using (\ref{eq_error})
      \end{itemize}
\item while ($error > eps$ and $t < maxIter$)
      \begin{itemize}
      \item update matrix: (\ref{eq_at_values})
      \item calculate row sums: (\ref{eq_rit_values})
      \item compute error: (\ref{eq_error})
      \item increase iteration number: $t \leftarrow t+1$
      \end{itemize}
\end{enumerate}  
\begin{remark}
As mentioned, Theorem \ref{theo_AoX} is also valid
using column sums. In the implementation, one can compute the row and column
sums and perform the next step using the sum where the initial error is the
lowest.
\end{remark}

\begin{remark}
From an iteration to the next, the error should decrease by an amount that
depends on the convergence rate. Otherwise, we must stop the algorithm
because the matrix does not seem to be primitive. This observation can be
used as an indirect test for primitivity of a matrix.
\end{remark}

Indeed, if the spectral radius is not simple (case of an irreducible 
imprimitive matrix) there may be at least two eigenvectors associated with 
eigenvalues having the same modulus.

\begin{remark}
With the proposed algorithm, the diagonal entries of $A$ remain unchanged,
only the off-diagonal components are modified. The Gerschgorin discs
\cite[page 388]{Horn_al_2019} associated with a matrix allow to illustrate
this. 
\end{remark}
Let us consider an example:
\begin{equation}
A = \left(\begin{array}{cc}3&\sqrt{3}\\ \sqrt{3} &1\end{array}\right)
\mbox{ ; } A^{(1)} = \left(\begin{array}{cc}3&1\\ 3 &1\end{array}\right)
\end{equation}
The off-diagonal components
of $A$ are modified in such a way all discs cross the same highest point
on the x-axis ($4$ for this example). To show this, let us write:
\begin{equation}
A = D_A + P\label{eq_DaP}
\end{equation}
where $D_A$ is a diagonal matrix formed with the diagonal
elements of $A$ and $P$ is matrix $A$ where the diagonal elements are set
to zero. From (\ref{eq_AoX}), we have:
\begin{equation}
B = D_A + P\circ X \label{eq_Da_PoX}
\end{equation}
Hence, the row sums of matrix $B$ are given by:
\begin{equation}
r_i(B) = r_i(D_A) + r_i(P\circ X)
       = a_{ii} + \mathbf{p}_{i.}^T\mathbf{x}_{i.}
\end{equation}
where $\mathbf{p}_{i.}$ and $\mathbf{x}_{i.}$ are vectors formed with
row $i$ of matrices $P$ and $X$, respectively.

\begin{remark}
Compared to the power method, there is no initial
vector to set. The results obtained using the power method are the 
Perron root and the associated eigenvalue. Only the
Perron root is obtained using this algorithm.
However, since the row (column) sums
of the matrix $B$ in (\ref{eq_AoX}) are the same, a vector $\mathbf{1}$
formed with only ones is an eigenvector of $B$ ($B^T$).
\begin{equation}
B\mathbf{1} = \rho(A)\mathbf{1} = (A\circ X)\mathbf{1} \label{eq_Bj}
\end{equation}
\end{remark}

\begin{remark}
The proposed algorithm uses another matrix in comparison with the power
method. At each iteration, the total numbers of
multiplications and additions of the matrix-vector multiplication by the
power method are equal to the total number of operations for the proposed
algorithm. Hence, using the power method, the additional arithmetic
operations used for calculating the eigenvector and the eigenvalue are extra
computational load compared to the proposed algorithm.
\end{remark}

\subsubsection{Algorithm B (using row sums): Perron root and vector
\label{algo_2}}
Instead of the algorithm A (\ref{algo_1}), another one can
consist in searching for a vector $\mathbf{y}$ similar to the Perron vector.
For this purpose, the matrix $B$ is initially
equal to $A$ and the vector $\mathbf{y}$ is set to $\mathbf{1}$,
then, the row sums of $B$ are calculated.
At iteration $t$, a vector $\mathbf{y}^{(t)}$
is formed using row sums and the matrix $B$ is updated.
At the convergence, we should have $\mathbf{y}^{(t)}\approx\mathbf{1}$.

\begin{enumerate}
\item Initialization
      \begin{itemize}
      \item set: $t\leftarrow 0$, $B\leftarrow A$,
            $\mathbf{y}\leftarrow\mathbf{1}$ 
            and compute row sums $r_i^{(0)}$ of $B$,
      \item calculate initial error: $\max(r_i^{(0)})-\min(r_i^{(0)})$.
      \item set stopping rules: $eps$ and $maxIter$ (see Algorithm A)
      \end{itemize}
\item while ($error > eps$ and $t < maxIter$)
      \begin{itemize}
      \item form $\mathbf{y}^{(t)}$ using row sums $r_i^{(t)}$
      \item update $\mathbf{y}$: $\mathbf{y}\circ \mathbf{y}^{(t)}$
      \item form $\mathbf{x}$ ($1/\mathbf{y}$) and update matrix
            $B$: $A \circ (\mathbf{x} * \mathbf{y}^T)$
      \item compute error: $\max|\mathbf{y}^{(t)}-\mathbf{1}|$
      \item increase iteration number: $t \leftarrow t+1$
      \item calculate row sums $r_i^{(t)}$ of $B$
      \end{itemize}
\end{enumerate}

The convergence test of this algorithm consists to have only ones for
vector $\mathbf{y}$. Instead, the
convergence test can be based on the examination of the minimum
and maximum row sum values of the decreasing sequences
$\xi^{(t)}$ and $\zeta^{(t)}$ defined in Theorem \ref{theo_convergence}.
The code in appendix \ref{ann_algoC}
is the R implementations of algorithm B using this test. An
input matrix should be square, nonnegative and
each row sum should be greater than zero.

\subsubsection{Application to row-stochastic matrices}
Let us consider Markov chains modeling a dynamic system with finite
discrete $n$ states. At each time $t$, this system is in one state. 
When the system is in state $i$, the move to state $j$ or to stay in 
state $i$ is controlled by the probability $p_{ij}\geq 0$, 
$\sum_{j=1}^np_{ij}=1$. The probabilities are organized in a transition 
matrix $P$ which is nonnegative. Interestingly, the power $k$ of matrix 
$P$ is also a transition matrix which element $p_{ij}$ corresponds to the
probability to move from state $i$ to state $j$ at time $t+k$. The Markov 
chains are used in: (a) biological sequences analysis
\cite{Durbin_al_2002,Ewens_al_2002}, (b) internet traffic or search
engines \cite{Langville_al_2005,Wu_al_2007,Wen_al_2017}, (c) \ldots
The state occupied by the Markov chain process at time $t+k$ depends on 
the nature of transition matrix $P$: reducible/irreducible, 
imprimitive/primitive. In some applications, the observed transition 
matrix is modified to be primitive \cite{Langville_al_2005}.
For this particular case, the power of the modified matrix $\hat{P}$
converges to a matrix formed with the same vector $\mathbf{u}$ verifying: 
$\mathbf{u}^T\hat{P} = \mathbf{u}^T$.
The vector $\mathbf{u}$ is the left-hand vector of the matrix $\hat{P}$.
In the search engines based on the Markov chains, the components of 
the vector $\mathbf{u}$ allow to hierarchize the information.
Applying algorithm B (\ref{algo_2}) to the transposed of the
row-stochastic (modified transition) matrix lead to a vector $\mathbf{y}$
which components are then normalized in a way that their sum is $1$. 
This normalized vector corresponds to $\mathbf{u}$.

\section{Results and conclusions} \label{sect_results_conclusions}
All calculations were performed on the same computer (a laptop equipped 
with i7-66600U processor, 16 GB of RAM, under Microsoft Windows 10) 
and R version 3.6.2. The default error level was arbitrarily set to 1.0E-8.

The first example is:
\begin{equation}
A = \left(\begin{array}{ccc} 2 &1 &0\\ 0.5 &3 &2\\ 1 &2 &4
    \end{array}\right) \label{eq_A_example2}
\end{equation}
The range values for the row and the column sums are $4$ and $2.5$
respectively. The algorithm is performed using the column sums.
\begin{figure}[!ht]
\centerline{\includegraphics[width=0.5\textwidth]{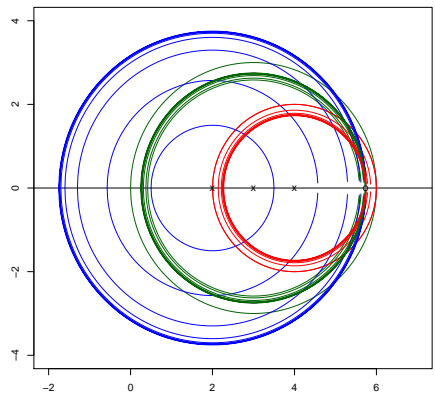}}
\caption{Gerschgorin's discs: thin plot lines for $A$ and $A^{(t)}$ before
convergence, bold plot lines for $A^{(t)}$ at convergence 
\label{fig_gersgorin_discs2}}
\end{figure}
Figure \ref{fig_gersgorin_discs2} presents the Gerschgorin discs for all
iterations. The Perron root for this example is $5.739952$,
the proposed algorithm and the power method
require $17$ and $19$ iterations, respectively. Figure 
\ref{fig_gersgorin_discs2} shows that the major modifications of the
off-diagonal elements of the matrix A are done during the first five
iterations.

For all of the tests performed, the algorithm proposed and the
power method have close number of iterations.
Worse results, in term of the number of iterations, were obtained using a
tridiagonal matrix.
Let $T(n ;c,a,b)$ a tridiagonal matrix of order $n$, where $a$ is the value 
for the diagonal components, $b$ is the value for the upper diagonal 
components and $c$ is the value for the under diagonal components.
An explicit expression relating eigenvalues of matrix
$T$ is available \cite{Noschese_al_2013}:
\begin{equation}
\lambda_k = a+2\sqrt{bc}\cos{\frac{k\pi}{n+1}}
\end{equation}
The ratio $\lambda_2/\lambda_1$ for $T$ is near $1$ when $n$ is high.
For $n=50$, $a=3$, $b=2$ and $c=1$ the first two eigenvalues of matrix
$T$ are: $5.823063$ and $5.806989$.
The proposed algorithm took $5,890$ (algorithm A) or $5,174$ (algorithm B)
iterations to calculate the Perron root.
The power method need $5,159$ iterations.

Except the cases where the modulus of the second eigenvalue is near to
that of the first, the algorithm proposed converges after few iterations,
especially when the first eigenvalue is largely dominant. The proposed 
method has been succesfully used for a matrix of order $15,515$ that results
from high-throughput biological data.

With a convergence rate similar to that of the classic power method,
the proposed algorithm for computing the Perron root is computationally less
demanding. But, it applies to only primitive matrices. However, it can 
be used as a low cost primitivity test compared to the matrix power
calculations involved in the Frobenius and Wielandt tests.

\subsection*{Acknowledgements}
This work was supported by funds from CNRS, INSERM and University of
Strasbourg.

Author is grateful to a referee for the valuable comments and suggestions.

\bibliographystyle{elsarticle-num}
\bibliography{pfro}

\appendix
\section{R code using row sums\label{ann_A}}
\subsection{Algorithm B\label{ann_algoC}}
\begin{verbatim}
## This function computes iteratively the Perron root
## and the eigenvector of matrix A using row sums
#
#  A: nonnegative square matrix (all row sums are > 0)
#  tol: error level used (stopping criterion)
#  maxIter: maximum number of iterations (stooping criterion)
#
# Returned
#  B: matrix which has the same row sums (B = A o X)
#  pfr: minimum and maximum row sum, that defines to the Perron root
#  y: vector (leading to have the eigenvector and matrix X)
#  iter: number of iterations performed
#  rmin: sequences with minimum row sum values for iterations
#  rmax: sequences with maximum row sum values for iterations
calcPRc <- function(A, tol=1.0e-8, maxIter=50) {
   n <- nrow(A);  m <- ncol(A)
   ri <- apply(A, 1, sum)
   ko <- (sum(A<0) || (min(ri)==0))
   if ((n != m) || (ko)) {
      stop("calcPRc(): for nonnegative primitive matrices")
   }
   y <- rep(1,n)
   iter <- 1; B <- A
   rmin <- c(); erMin <-  rmin[iter] <- min(ri)
   rmax <- c(); erMax <-  rmax[iter] <- max(ri)
   erIter <- ((erMin > tol) || (erMax > tol))
   while (erIter && (iter < maxIter)) {
         yt <- ri/ri[1]; y <- y*yt
         B <- A * ((1/y) %*% t(y))
         ri <- apply(B, 1, sum)
         iter <- iter + 1
         rmin[iter] <- ri.min <- min(ri)
         rmax[iter] <- ri.max <- max(ri)
         erMin <- rmin[iter] - rmin[iter-1]
         erMax <- rmax[iter-1] - rmax[iter]
         erIter <- ((erMin > tol) || (erMax > tol))
   }
   pfr <- c(ri.min, ri.max)
   list(B=B, pfr=pfr, y=y, iter=iter-1, rmin=rmin, rmax=rmax)
}
\end{verbatim}

\end{document}